\documentclass[12pt,journal,onecolumn]{IEEEtran}
\usepackage{amsmath,amssymb,amsthm,dsfont}
\usepackage[pdftex]{color}
\usepackage[normalem]{ulem}

\usepackage{setspace}
\newtheorem{theorem}{\bf Theorem}[section]
\newtheorem{proposition}{\bf Proposition}[section]

\newtheorem{remark}{\bf Remark}[section]
\newtheorem{example}{\bf Example}

\newtheorem{algorithm}{\bf Algorithm}

\newtheorem{problem}{\bf Problem}[section]

\newcommand{\bmat}{\left[ \begin{matrix}}
\newcommand{\emat}{\end{matrix} \right]}

\newcommand{\N}{\mathbb  N}
\newcommand{\R}{\mathbb  R}

\newcommand{\bC}{\mathbf C}

\newcommand{\bx}{\mathbf x}

\newcommand{\bSigma}{\boldsymbol{\Sigma}}

\newcommand{\cK}{\mathcal K}

\newcommand{\cS}{\mathcal S}

\newcommand{\Ebb}{{\mathbb E}\,}

\definecolor{darkgreen}{rgb}{0.0, 0.7, 0.0}

\begin{document}

\title{Maximum Entropy Kernels for System Identification}

\author{Francesca~P.~Carli,
        Tianshi~Chen, ~\IEEEmembership{Member,~IEEE,}
        and~Lennart~Ljung ~\IEEEmembership{Life~Fellow,~IEEE}
\thanks{This work was supported by the ERC advanced grant LEARN No. 267381, funded by the European Research Council.}
\thanks{Francesca~P.~Carli is with the Department of Engineering, University of Cambridge, Cambridge, United Kingdom,
 and with the Department of Electrical Engineering and Computer Science, University of Li\`{e}ge, Li\`{e}ge, Belgium,
{\tt\small  fpc23@cam.ac.uk}.
}
\thanks{Tianshi~Chen and Lennart~Ljung are with the Division of Automatic Control, Department of Electrical Engineering,
Link\"{o}ping University, Link\"{o}ping, Sweden,
{\tt\small  tianshi.chen@liu.se}, {\tt\small  ljung@isy.liu.se}.
}
}

\maketitle

\begin{abstract}
A new nonparametric approach for system identification has been recently proposed
where the impulse response is modeled as the 
realization of a zero--mean Gaussian process whose covariance
(kernel) has to be estimated from data. In this scheme, quality of
the estimates crucially depends on the parametrization of the
covariance of the Gaussian process. A family of kernels that have
been shown to be particularly effective in the system identification
framework is the family of Diagonal/Correlated (DC) kernels. Maximum
entropy properties of a related family of kernels, the Tuned/Correlated (TC) kernels, have been recently pointed out in the literature.
In this paper we show that maximum entropy properties indeed extend  
to the whole family of DC kernels. The maximum entropy
interpretation can be exploited in conjunction with results on
matrix completion problems in the graphical models literature
to shed light on the structure of the DC kernel. In particular, we
prove that the DC kernel admits a closed--form factorization, inverse and determinant. 
These results can be exploited both to improve the numerical stability and to
reduce the computational complexity associated with the computation
of the DC estimator.
\end{abstract}

\section{Introduction}\label{sec:Introduction}

{System identification is concerned with automatic dynamic model building from measured data.
``Classical'' approaches to linear system identification are parametric maximum likelihood (ML)/Prediction Error Methods (PEM).
Here, a finitely--parametrized model class (OE, ARX, ARMAX, Box-Jenkins, etc.)  is first specified, then
model--order selection is performed by
optimizing some penalized goodness--of--fit criteria, such as the
Akaike information criterion (AIC) 
or the Bayesian information
criterion (BIC), 
or via cross validation (CV) 
and finally the finite dimensional parameter vector is estimated by minimizing a suitable cost function (the prediction error or minus the log--likelihood).
While ``standard'' (i.e. without model selection) statistical theory suggests that PEMs should be asymptotically efficient for Gaussian innovations,
sample properties of the estimates returned after model selection, the so--called Post--Model--Selection--Estimators (PMSE),
may much depart from those suggested by asymptotic analysis. }

{Inspired by works in
statistics \cite{Tibshirani1996}, machine learning \cite{ShaweTaylorCristianini2004},
and inverse problems \cite{TikhonovArsenin1977},
regularized kernel methods have been recently introduced in the system identification scenario \cite{PillonettoDeNicolao2010, ChenOhlssonLjung2012, PillonettoDinuzzoChenDeNicolaoLjung2014} to jointly perform model selection  and  estimation. 
According to these methods, the estimated impulse response is seen as the solution of a variational problem whose cost function is the combination of two terms:
a loss function that takes into account fit to the data
and a regularization term that keeps under control the complexity of the solution. 
In this paper, we consider $\ell_{2}$--type penalties which also admit a Bayesian interpretation in the form of a zero--mean Gaussian prior imposed on the to-be-estimated impulse response.
The corresponding Bayesian procedure goes under the name of Gaussian process regression.
The covariance of the Gaussian process usually depends on few hyperparameters that, according to an empirical Bayes paradigm, are learnt from data via e.g. marginal likelihood (ML) maximization.
The resulting Bayesian nonparametric approach equipped with marginal likelihood estimation of the hyperparameters
has shown advantages over parametric PEMs
above all in situations where the number of measurements and/or the signal to noise ratio is not particularly high
(see \cite{PillonettoDeNicolao2010,ChenOhlssonLjung2012, CarliChenChiusoLjungPillonetto2012} 
and \cite{PillonettoDinuzzoChenDeNicolaoLjung2014} where the scheme has been successfully tested on a real robotic application). }

{In this scheme, quality of the estimates crucially depends on the choice of the covariance (kernel) of the Gaussian process,
which should be flexible enough to capture diverse and complicated dynamics,
but simple enough to keep the ML maximization problem easy to solve.
Many kernels have been introduced in the machine learning literature 
\cite{ShaweTaylorCristianini2004}, but
a straight application of these kernels
in the system identification framework is doomed to fail mainly because they fail to encode desirable structural properties (e.g. BIBO stability) for
the estimated impulse response.
For this reason, several kernels
have been introduced in the system identification literature
\cite{PillonettoDeNicolao2010,ChenOhlssonLjung2012}.
}

{This paper focuses on Diagonal/Correlated (DC) kernels, which were introduced by a deterministic argument in
\cite{ChenOhlssonLjung2012}, where their performance was assessed on a data bank of test systems and data sets.
In this paper, we show that DC kernels admit a maximum entropy interpretation.
Maximum entropy properties of a related family of kernels, the TC kernels, have been pointed out in the literature, see \cite{Carli2014,PillonettoDeNicolao2011}.
Here we show that arguments in \cite{Carli2014}, resting on the theory of matrix completion
problems in the graphical models literature \cite{Dempster1972,GroneJohnsonSaWolkowicz1984,DymGohberg1981,Johnson1990,GohbergGoldbergKaashoek1993,DahlVanderbergheRoychowdhury2008},
extend to the whole family of DC kernels.
While TC kernels can be seen both as a particularization of DC kernels for $\rho=\sqrt{\lambda}$ \cite{ChenOhlssonLjung2012} 
(see below) and as a member of
the family of stable spline kernels  \cite{PillonettoDeNicolao2010}, 
it is interesting to note that
the maximum entropy interpretation does \emph{not} extend to the family of stable spline kernels (e.g. to second--order stable spline kernels). }

{
Leveraging on the theory of matrix completion problems, the maximum entropy interpretation 
is used to shed light on the structure on the DC kernel, leading to a
closed form expression for the inverse and determinant,
as well as to a closed form factorization.
Since best state--of--the--art algorithms for marginal likelihood maximization are based on numerical computation of the Cholesky factors of the kernel matrix, which can be very ill--conditioned,
the derived closed form factorization can be used to improve \emph{numerical stability} of these procedure.
By exploiting sparsity of the inverse of the kernel matrix, \emph{computational efficiency} can also be significantly improved, as detailed in the paper.   }

%
%
\textbf{Notation.}
In the following, we denote by $\cS_n$ the set of symmetric matrices of order $n$.
We write $A \succeq 0$ (resp. $A \succ 0$) to denote that $A$ is positive semidefinite (resp. positive definite).
$I_k$ will denote the identity matrix of order $k$,
and ${\rm diag}\left\{a_1, a_2,\dots, a_k\right\}$ the diagonal matrix with diagonal elements $\left\{a_1,a_2, \dots, a_k\right\}$.
If $A$ is a square matrix of order $n$, for index sets $\beta \subseteq \left\{1, \dots, n\right\}$ and $\gamma \subseteq \left\{1, \dots, n\right\}$,
we denote by $A(\beta,\gamma)$ the submatrix that lies in the rows of $A$ indexed by $\beta$ and the columns indexed by $\gamma$.
If $\gamma=\beta$, we use the shorthand $A(\beta)$ for $A(\beta, \beta)$.
Finally, for $p,q\in\N$ with $p<q$, the index set $\{p,p+1,\cdots,q\}$ is abbreviated as $p:q$.

\section{Linear system identification via Gaussian Process Regression}\label{sec:LinearSYSDviaGP}


{Consider a single input single output (SISO) system described by } 
\begin{align}\label{eq:dyn}
y(t) = \sum_{k=1}^\infty g(k) u(t-k) + v(t), \quad t=1,\dots, N
\end{align}
where $u(t)$ and $y(t)$ are the measurable input and output signal, respectively,
$g(t)$ is the (unknown) impulse response and
$v(t)$ is a zero--mean white Gaussian noise with variance $\sigma^2$.
The system identification problem amounts to the following: given $\{y(t),u(t)\}_{t=1}^N$ find and estimator of the impulse response $g$.

Model (\ref{eq:dyn}) can be expressed in matrix form as
\begin{equation}\label{eq:linreg}
Y= \Phi^\top g + V
\end{equation}
where $Y$ and $V$ are $N$--dimensional vectors, $\Phi^\top \in \R^{N \times \infty}$ is a matrix whose entries are defined by the system input,
the unknown samples from $u(t)$ having been set to zero, and
$g=\bmat g(1)\ g(2)\ \dots \; \;  \; \emat^\top$ is the infinite--dimensional vector  representing the impulse response.


\subsection{Gaussian process regression with DC Kernels}\label{subsec:GPregression_numerical stability via_SSkernel}

Under the framework of Gaussian process regression
\cite{RasmussenWilliams2006}, the impulse response $g(t)$ is modeled
as a discrete--time Gaussian process with zero mean and a suitable
covariance, independent of $v(t)$:
\begin{equation}\label{modgp}
g(t) \sim \mathcal{GP}(0, K(t,t';\eta)),\quad t,t'=1,2,\dots \,.
\end{equation}
The covariance $K(t,t'; \eta)$ (also referred to as the kernel in the machine learning literature)
is parameterized by a vector of hyperparameters $\eta$.
Since the impulse response of linear stable systems decays exponentially to zero,
it is often enough to truncate the infinite impulse response at a certain order $n$
so that $g$ in (\ref{eq:linreg}) becomes
$g=[g(1)\ g(2)\ \cdots\ g(n)]^\top$, and
$\Phi^\top \in \R^{N \times n}$;
likewise, $K(\eta)$ in \eqref{modgp} becomes a $n \times n$ covariance (kernel) matrix.
{In a system identification scenario, the number of data $N$ can be of the order of several thousands while the truncation order $n$ shouldn't exceed a couple of hundreds.
It follows that we shall typically assume $n \ll N$.  }

\vspace{2mm}
According to an Empirical Bayes paradigm \cite{MaritzLwin1989}, 
the hyperparameters can be estimated from the data via marginal likelihood maximization, i.e.
by maximizing the marginalization with respect to $g$ of the joint density of $Y$ and $g$
\begin{align}\label{eq:lglklhd} 
\hat \eta 
&= \arg\min_{\eta} \;\left\{\log \det (\Phi^\top K(\eta)
\Phi+\sigma^2 I_N)\right.\nonumber\\ &\qquad\qquad\qquad\qquad\left.
+ Y^\top (\Phi^\top K(\eta) \Phi + \sigma^2 I_N)^{-1} Y\right\}\,.
\end{align}

{The estimated impulse response is computed
as the maximum a posteriori (MAP) estimate (that, for symmetric and unimodal probability density functions, coincides with the conditional mean), which,  once $\eta$ is estimated, admits the closed form representation}
\begin{equation}\label{eqn:f}
\hat g_{MAP} \equiv  \Ebb \left[g | Y, \hat \eta \right] = K(\hat \eta) \Phi
\left(\Phi^\top K(\hat \eta) \Phi + \sigma^2 I_N\right)^{-1} Y \,.
\end{equation}
Prior information is introduced in the identification process by
assigning the covariance $K(\eta)$.
The quality of the estimates
crucially depends on this choice as well as on the quality of the
estimated $\hat \eta$. Several kernels have been recently introduced
in the system identification literature
\cite{PillonettoDeNicolao2010,ChenOhlssonLjung2012}.
A class of prior covariances which has been proved to be very
effective in the system identification scenario, is the class of
Diagonal/Correlated (DC) kernels \cite{ChenOhlssonLjung2012}:
\begin{equation}\label{eqn:DC_kernel}
     \left(\cK_{DC}\right)_{ij} := \cK_{DC}(i,j;\eta) = c \, \lambda^{\frac{i+j}{2}} \rho^{\left|i-j\right|},\ i,j=1,\cdots,n 
\end{equation} where $\eta=[c\ \lambda\ \rho]^\top$
with $c\geq 0$ ,  $0 \leq \lambda < 1$, $-1 \leq \rho \leq 1$. Here
$\lambda$ accounts for the exponential decay of the impulse
response, while $\rho$ describes the correlation between neighboring
impulse response coefficients. If $\rho = \sqrt{\lambda}$, we obtain
the so--called Tuned/Correlated (TC) kernel  \cite{ChenOhlssonLjung2012}
\begin{equation}\label{eqn:TC_kernel}
    \left(\cK_{TC}\right)_{ij} := \cK_{TC}(i,j;\eta) = c \lambda^{max(i,j)}, \ c \geq 0, \, 0 \leq \lambda < 1\,.
\end{equation} which has also been introduced via a stochastic argument under the name of ``first--order stable spline (SS) kernel'' in \cite{PillonettoDeNicolao2011}.

{
\begin{remark}\label{rem:marginal-lik-maximization_discussion}
Hyperparameters can be estimated from data in a variety of ways;
If on one hand empirical evidence \cite{ChenOhlssonLjung2012, PillonettoDinuzzoChenDeNicolaoLjung2014}
as well as some theoretical results \cite{CarliChenChiusoLjungPillonetto2012} 
support the use of 
marginal likelihood maximization, this boils down to a challenging nonconvex optimization problem.
Indeed, no matter what nonconvex optimization solver is used,
tuning of the hyperparameters requires repeated computations of the marginal likelihood (for many different hyperparameters values),
whose straightforward computation requires the inversion of a $N \times N$ matrix (with a computational complexity  that scales as $O(N^{3})$)
which may be impracticable in data--rich situations.
Also numerical stability can be an issue, since computation of the marginal likelihood involves extremely ill--conditioned matrices (see \cite{ChenLjung2013} for details).
For this reason, several dedicated algorithms have been developed in the literature \cite{CarliChiusoPillonetto2012SYSID, ChenLjung2013}.
Best state--of--the--art algorithms scale as $O(n^{3})$ (recall that here $n$ is the order of the finite impulse response truncation, with $n \ll N$)
but in these algorithms, computation of the objective function as well as of its gradient and Hessian invariably requires the factorization of the kernel matrix.
Indeed, $\cK_{DC}$ has an exponentially decaying diagonal and can be very ill--conditioned if the parameters that control the decay rate are very small.
For example \cite{ChenLjung2013}, if $n=125$, the $2$--norm condition number of the DC kernel, is $2.99 \times 10^{8}$ for $\lambda=0.9$, $\rho=0.98$ and $3.84 \times 10^{29}$ for $\lambda=0.6$, $\rho=0.98$.
Closed form  factorization of the kernel matrix as well as sparsity of its inverse  as derived in Section \ref{sec:DC_as_MaxEntrkernel} can be used both to improve numerical stability  and reduce the
computational burden associated to the computation of  the marginal likelihood cost function and its gradient and Hessian.
\end{remark}
}

\section{Maximum Entropy band extension problem}\label{sec:MaxEntrBandExt}

{In this section, we introduce relevant theory about the maximum entropy covariance extension problems that will be
used to prove our main results in Section \ref{sec:DC_as_MaxEntrkernel}. }
Covariance extension problems were introduced by A.~P. Dempster
\cite{Dempster1972} and studied by many authors (see e.g.
\cite{GroneJohnsonSaWolkowicz1984,DymGohberg1981,Johnson1990,GohbergGoldbergKaashoek1993,DahlVanderbergheRoychowdhury2008}
and references therein, see also \cite{CFPP-2011,CFPP-2013}
for an extension to the circulant case).
{In this framework, it is customary to associate
the pattern of the specified entries of an $n\times n$ partially specified symmetric matrix to an undirected graph with $n$ vertices.
This graph has an edge joining vertex $i$ and vertex $j$ if and only if the $(i,j)$ entry is specified.
If the graph of the specified entries is \emph{chordal} (i.e., a graph in which every cycle of
length greater than three has an edge connecting nonconsecutive nodes),  
then the maximum entropy covariance extension problem admits a
closed form solution in terms of the principal minors of the matrix
to be completed (see, e.g., \cite{Lauritzen1996}). } 
{In the following, we introduce the maximum entropy covariance extension problem
specializing to the case where the given entries lie on a band centered along the main diagonal.
As it is easy to see, the associated undirected graph is chordal and the maximum entropy covariance extension problem
admits a closed form solution that can be computed in a recursive way. }

\subsection{Differential Entropy}

Recall that the {\em differential entropy} $H(p)$ of a probability
density function $p$ on $\R^n$ is defined by
\begin{equation}\label{DiffEntropy}
H(p)=-\int_{\R^n}\log (p(x))p(x)dx.
\end{equation}
In  case of a zero--mean Gaussian distribution $p$ with covariance
matrix $\bSigma_n$, we get
\begin{equation}\label{gaussianentropy}
H(p)=\frac{1}{2}\log(\det\bSigma_n)+\frac{1}{2}n\left(1+\log(2\pi)\right).
\end{equation}

\subsection{Problem Statement}

We will discuss $n\times n$ matrices $\bSigma(\bx)$ where only certain elements are given and fixed, and we want to assign
the values of the remaining elements $\bx$ in such a way that $\bSigma(\bx)$
becomes a positive definite matrix with as large value as possible of $\det \bSigma(\bx)$.
In view of \eqref{gaussianentropy} this is called the \emph{maximum entropy extension problem}.

We shall in particular study the problem that the fixed values are given \emph{in a band} around the main diagonal of $\bSigma$.

\begin{problem}[\textbf{Maximum entropy band extension problem}] Let $\bSigma(\bx)$ be $n \times n$ matrix and $i=1,\dots,n$,  $j=1,\dots,n$.
For given $m$, $\sigma_{ij}$, $|i-j|\leq m$, the maximum entropy band extension problem is
\begin{subequations}\label{probl:MaxEntr}
\begin{eqnarray}
\underset{\bx}{{\rm minimize}} & \left\{ -\log \det \bSigma(\bx) \mid \bSigma(\bx) \in \cS_n \right\} \\
\text{\emph{subject to }} & \bSigma(\bx) \succ 0 \\
& \left[\bSigma(\bx)\right]_{ij} = \sigma_{ij}, \quad |i-j|\leq m
\end{eqnarray}
\end{subequations}
\end{problem}

To stress the dependence on $n$ and $m$ we shall use the notation $\bSigma^{(m)}_n$ and suppress the argument $\bx$.
The problem is a convex optimization problem since we are minimizing a strictly convex function on the intersection of a convex cone (minus the zero matrix) with a linear manifold.
We will denote by $\bx^o$ its optimal value and by $\bSigma_n^{(m),o}\equiv\bSigma_n^{(m)}(\bx^o)$ the associated optimal extension.

{
\begin{remark}
As explained by Dempster in his seminal work \cite{Dempster1972}, ``the principle of seeking maximum entropy is a principle of seeking maximum simplicity of explanation''.
The inverse of the maximum entropy extension has zeros in the complementary position of those assigned (see Theorem \ref{thm:feasibility_and_bandedness}, point (ii), below).
As the elements of the inverse of the covariance matrix (aka the precision matrix) represents the natural parameters of the Gaussian distribution, 
the maximum entropy completion obeys a ``principle of parsimony in parametric model fitting, which suggests that parameters should be introduced sparingly
and only when the data indicate they are required'' \cite{Dempster1972}.
Notice that 
setting the $(i,j)$--th element of the inverse covariance matrix to zero has the probabilistic  interpretation that the $i$--th and $j$--th components of the underlying Gaussian random vector are conditionally independent given the other components.
This choice may at first look less natural than setting the unspecified elements of $\bSigma$ to zero. It has nevertheless considerable advantages compared to the latter, cf. \cite[p. 161]{Dempster1972}.
\end{remark}
}

\subsection{Basic results for banded matrices}

\begin{theorem}[\cite{Dempster1972, DymGohberg1981}] 
\label{thm:feasibility_and_bandedness}
\begin{itemize}
    \item[$(i)$] \emph{Feasibility:} Problem \eqref{probl:MaxEntr} is feasible,
    namely $\bSigma_n^{(m)}$ admits a positive definite extension if and only if
    \begin{equation}\label{eqn:CNS_A_admits_posdef_ext}
    \begin{bmatrix}
    \sigma_{i,i} & \cdots & \sigma_{i,m+i} \\
    \vdots & & \vdots \\
    \sigma_{m+i,i} & \cdots & \sigma_{m+i,m+i}
    \end{bmatrix}
    \succ 0,
    \;\,\,
    i=1, \ldots, n-m
    \end{equation}
    \item[$(ii)$] \emph{Bandedness: }
    Assume \eqref{eqn:CNS_A_admits_posdef_ext} holds.
    Then    \eqref{probl:MaxEntr} admits a unique solution with the additional property that its inverse is banded
    of bandwidth $m$, namely the $(i,j)$--th entries of $\left(\bSigma_n^{(m),o}\right)^{-1}$
    are zero for $|i-j|>m$.
\end{itemize}
\end{theorem}
The positive definite maximum entropy extension $\bSigma_n^{(m),o}$
is also called \emph{central extension} of $\bSigma_n^{(m)}$.

For banded sparsity patterns like those considered so far, Problem
\eqref{probl:MaxEntr} admits a closed form solution that can be
computed recursively in the following way.
We start by considering a partially specified $n \times n$ symmetric matrix of bandwidth $(n-2)$
\begin{equation}\label{eqn:def_Sigma0}
\bSigma_{n}^{(n-2)} = \bmat
\sigma_{1,1} & \sigma_{1,2} & \ldots & \sigma_{1,n-1} & x \\
\sigma_{1,2} & \sigma_{2,2} & \ldots & \sigma_{2,n-1} & \sigma_{2,n} \\
\vdots & \vdots  & & \vdots & \vdots \\
\sigma_{1,n-1} & \sigma_{2,n-1} & \ldots & \sigma_{n-1,n-1} & \sigma_{n-1,n} \\
x & \sigma_{2,n} & \ldots & \sigma_{n-1,n} & \sigma_{n,n}\\
\emat
\end{equation}
and consider the submatrix
\begin{equation}\label{eqn:LMR}
L = [\sigma_{ij}]_{i,j = 1}^{n-1}\,. 
\end{equation}
We call \emph{one--step extension} the extension of a $n \times n$,
$(n-2)$--band matrix. The next theorem gives a recursive algorithm
to compute the extension
of a partially specified matrix of generic bandwidth $m$ by
computing the one--step extension
of suitable submatrices. It also gives a representation of the solution in factored form.

\begin{theorem}[\cite{GohbergGoldbergKaashoek1993}, \cite{DymGohberg1981}]
\label{thm:closed_form_central_extension}
\begin{itemize}
    \item[$(i)$] The one--step central extension of $\bSigma_{n}^{(n-2)}$ is given by
 \begin{equation}\label{eqn:z0}
x^o = - \frac{1}{y_1} \sum_{j=2}^{n-1} \sigma_{nj} y_j
\end{equation}
with
\begin{equation}\label{eqn:y}
\bmat  y_1& y_2 & \dots & y_{n-1} \\ \emat^\top = L^{-1} \bmat 1 & 0 &
\dots & 0  \emat^\top \,.
\end{equation}
Let $\bSigma_n^{(m)}$ be an $n \times n$ partially specified
$m$--band matrix.
The central extension $\bC = [c_{ij}]_{i,j=1}^{n} $ of
$\bSigma_n^{(m)}$ is such that
for all $m+1 < t \leq n$ and $1 \leq s \leq t-m-1$ the submatrices
$C(s:t)$
are the central one--step extensions of the corresponding
$(t-s-1)$--band matrix.

    \item[$(ii)$] In particular, the central extension of the partially specified symmetric $m$-band matrix $\bSigma_n^{(m)}$ admits the factorization
\begin{equation}\label{eqn:central_extension_factored_form}
\bC = \left(L_n^{(m)}V U_n^{(m)}\right)^{-1}
\end{equation}
where $L_n^{(m)} = \left[\ell_{ij}\right]$ is a lower triangular
banded matrix with ones on the main diagonal, $\ell_{jj} = 1$, for
$j=1, \dots, n$, and
\begin{equation}\label{eqn:Xm}
\bmat \ell_{\alpha j}\\ \vdots \\\ell_{\beta j} \emat =
- \bmat \sigma_{\alpha \alpha} & \dots & \sigma_{\alpha \beta}\\
\vdots & \ddots  &  \vdots\\
\sigma_{\beta \alpha} & \dots & \sigma_{\beta \beta} \emat^{-1}
\bmat \sigma_{\alpha j}\\ \vdots \\\sigma_{\beta j} \emat 
\end{equation}
for $j=1, \dots, n-1$, $U_n^{(m)} = \left(L_n^{(m)} \right)^\top$
and $V=\left[v_{ij}\right]$  diagonal with entries
\begin{equation}\label{eqn:V}
v_{jj} = \left(\bmat \sigma_{jj} & \dots & \sigma_{j \beta}\\
\vdots & \ddots  &  \vdots\\
\sigma_{\beta j} & \dots & \sigma_{\beta \beta}
\emat^{-1}\right)_{1,1}\,, 
\end{equation}
for $j=1, \dots, n$, where 
$$
\begin{array}{lll}
\alpha = \alpha(j) = j+1 & \text{ for } & j=1, \dots, n-1\,,\\
\beta = \beta(j) = \min(j+m,n) & \text{ for } & j = 1, \dots, n\,.
\end{array}
$$
\end{itemize}
\end{theorem}


\section{Maximum Entropy properties of the DC kernel}\label{sec:DC_as_MaxEntrkernel}

{We provide a proof of the maximum entropy property
of the DC kernel. In particular, we prove that the DC kernel
maximizes the entropy among all kernels that satisfy the moment
constraints \eqref{eqn:moment_contraints_DC}.
The proof relies on the theory of matrix extension problems in Section \ref{sec:MaxEntrBandExt}.
This argument naturally leads to a closed form expression for the inverse of the DC kernel as well as to a closed--form factorization and determinant both for the kernel and its inverse.
}

\begin{proposition}\label{prop:MaxEntr_for_DCkernel}
Consider Problem \eqref{probl:MaxEntr} with $m=1$ and
\begin{equation}\label{eqn:moment_contraints_DC}
\sigma_{ij} = \left(\cK_{DC}\right)_{ij} = \rho^{\left|i-j\right|}
\lambda^{\frac{i+j}{2}}, \quad |i-j|  \leq 1
\end{equation}
i.e. consider the partially specified $1$--band matrix
    $$
    \bSigma_n^{(1)}(\bx) = \bmat
    \lambda & \lambda^{\frac32} \rho & x_{13} & \dots & \dots & x_{1n}\\
    \lambda^{\frac32} \rho & \lambda^2 & \lambda^{\frac52} \rho & x_{24}  & \dots & x_{2n}\\
    x_{13} & \lambda^{\frac52} \rho & \lambda^3 & \lambda^{\frac72} \rho&  & \vdots \\
    \vdots &  & \ddots & \ddots & \ddots & x_{n-2,1}\\
    \vdots &  &  & \lambda^{\frac{2n-3}{2}} \rho & \lambda^{n-1} & \lambda^{\frac{2n-1}{2}} \rho\\
    x_{1n} & \dots & \dots & x_{n-2,1} & \lambda^{\frac{2n-1}{2}} \rho & \lambda^n
    \emat
    $$
Then $\bSigma_n^{(1)}(\bx^o) = \cK_{DC}$, i.e. the solution of the
Maximum Entropy Problem \eqref{probl:MaxEntr} coincides with the DC
kernel \eqref{eqn:DC_kernel}.
\end{proposition}

\begin{proof}
    By Theorem \ref{thm:closed_form_central_extension},
    the maximum entropy completion of $\bSigma_n^{(1)}(\bx)$ can be recursively computed
    starting from the maximum entropy completions of the nested principal submatrices of smaller size.
    The statement can thus be proved by induction on the dimension $n$ of the completion.
    \begin{itemize}
        \item Let $n=3$, then by \eqref{eqn:z0}, \eqref{eqn:y}, the central extension of
    $$
    \bmat \lambda & \lambda^{\frac32}\rho & x_{13}\\
    \lambda^{\frac32}\rho &\lambda^2 & \lambda^{\frac52}\rho \\
    x_{13} & \lambda^{\frac52}\rho &\lambda^3\emat
    $$
    is given by $x_{13}^o= \lambda^2 \rho^2 = \cK_{DC}(1,3)$, as claimed.

    \item  Now assume that the statement holds for $n=k$, $k \geq 3$, i.e. that
    $\cK_{DC}(\left\{1,\dots, k\right\})$ is the central extension of $\bSigma_n^{(1)}(\left\{1,\dots, k\right\})$.
    We want to prove that $\cK_{DC}(\left\{1,\dots, k+1\right\})$
    is the central extension of $\bSigma_{n}^{(1)}(\left\{1,\dots, k+1\right\})$.
    To this aim, we only need to prove that the $(k-1)$ submatrices $\cK_{DC}(\left\{s,\dots, k+1\right\})$, $1 \leq s \leq k-1$,
    are the central one--step extensions of the corresponding $(k-s)$--band matrices
  $$
    \bmat
    \lambda^s & \dots & \lambda^{\frac{s+k}{2}\rho^{|k-s|}} & x_{s,k+1}\\
     \vdots & \ddots &  & \lambda^{\frac{s+k+2}{2}\rho^{|k-s|}}\\
     &  & \ddots & \vdots \\
    x_{s,k+1} & \lambda^{\frac{s+k+2}{2}\rho^{|k-s|}} & \dots & \lambda^{k+1}
    \emat\,,
    $$
    or, equivalently, that $x_{s,k+1}^o = \cK_{DC}(s,k+1)$, for $s=1, \dots, k-1$.
    In order to find $x_{s,k+1}^o$, we consider \eqref{eqn:y}, which, by the inductive hypothesis, becomes
    \begin{equation*}
    \bmat  y_1^{(s,k+1)}\\ y_2^{(s,k+1)} \\ \vdots \\ y_{k-s+1}^{(s,k+1)} \\ \emat
    = \cK_{DC}(\left\{s, \dots, k\right\})^{-1} \bmat 1 \\ 0 \\ \vdots \\ 0 \\ \emat \,.
    \end{equation*}
    By considering the adjoint of $\cK_{DC}(\left\{s, \dots, k\right\})$ 
    one can see that
    $y_2^{(s,k+1)} = - \frac{\rho}{\sqrt{\lambda}} y_1^{(s,k+1)}$
    while all the others $y_{i}^{(s,k+1)}$, $i=3, \dots, k$ are identically zero. 
    It follows that
    \begin{align*}
    x_{s,k+1}^o &= -\frac{1}{y_1^{(s,k+1)}} \, y_2^{(s,k+1)} \lambda^{\frac{s+k+2}{2}} \rho^{k-s} \\
    &=  \frac{\rho}{\sqrt{\lambda}} \lambda^{\frac{s+k+2}{2}} \rho^{k-s}
    = \cK_{DC}(s,k+1).
    \end{align*}
\end{itemize}\end{proof}

\begin{proposition}\label{prop:strucproperty4DCkernel}
{The DC kernel \eqref{eqn:DC_kernel}  satisfies the properties: }

\begin{itemize}
    \item[$(i)$] $\cK_{DC}$ admits the factorization
\begin{equation}\label{eqn:fatt_KDC}
\cK_{DC} = U W U^\top
\end{equation}
with $U$ upper triangular and Toeplitz given by
\begin{equation}\label{eqn:U}
U = \begin{cases}
1 & \text{if } \; i=j\\
0 & \text{if } \; i > j \\
\frac{\rho^{j-i}}{\lambda^{(j-i)/2}} & \text{ \emph{otherwise}}
\end{cases}
\end{equation}
and $W$ diagonal given by
\begin{equation}\label{eqn:W}
W = {(1-\rho^2)}{\rm diag}\left\{\lambda,\lambda^2,\dots
,\lambda^{n-1},\frac{\lambda^n}{1-\rho^2}\right\}\,.
\end{equation}

\item[$(ii)$]
$\cK_{DC}^{-1}$ is a tridiagonal banded matrix given by
\begin{equation}\label{eqn:KDC_inv}
\left(\cK_{DC}^{-1}\right)_{i,j} = \frac{c_{ij}}{1-\rho^2}
(-1)^{i+j} \lambda^{-\frac{i+j}{2}} \rho^{\left|i-j\right|}
\end{equation}
where
\begin{equation}\label{eqn:cij_KDC_inv}
c_{ij} = \begin{cases}
0 & \text{ if } \left|i-j\right|>1,\\
1+\rho^2 & \text{ if } i=j=2, \dots, n-1,  \\
1 & \text{\emph{ otherwise. }}
\end{cases}
\end{equation}
{and factors as 
\begin{equation}\label{eqn:sol_MaxEnt_per_DCkernel}
\cK_{DC}^{-1} = L V L^{\top}
\end{equation}
with $L$ Toeplitz lower triangular and banded, given by
\begin{equation}\label{eqn:Lm_DCkernel}
L = \begin{cases}
1 & \text{if } \; i=j \\
-\frac{\rho}{\sqrt{\lambda}} & \text{if } \; j=i-1 \\
0 & \text{otherwise}
\end{cases}
\end{equation}
and $V$ diagonal and given by
\begin{equation}\label{eqn:V_DCkernel}
V = \frac{1}{(1-\rho^2)}{\rm diag}\left\{\frac{1}{\lambda}\,,
\frac{1}{\lambda^2}\,, \dots, \frac{1}{\lambda^{n-1}}\,,
\frac{1-\rho^2}{\lambda^{n}}\right\}\,.
\end{equation}}

\item[$(iii)$] 
{The DC kernel has determinant
\begin{equation}\label{eqn:det_K_DC}
{\rm det}\left(\cK_{DC}\right) = \lambda^{\frac{n(n+1)}{2}}
\left(1-\rho^2\right)^{n-1}\,.
\end{equation}}
\end{itemize}
\end{proposition}
{
\begin{proof} The proof of points (i) and (ii) follows from the maximum entropy interpretation of the DC kernel (Proposition \ref{prop:MaxEntr_for_DCkernel}) by exploiting the factored representation of the solution of the maximum entropy band extension problem illustrated in Theorem \ref{thm:closed_form_central_extension}.
For what concern point (iii), it suffices to notice that the first (resp., third) factor on the right hand side of \eqref{eqn:sol_MaxEnt_per_DCkernel} is a lower (resp., upper) triangular matrix with diagonal entries equal to one, and hence 
$\cK_{DC}^{-1}$ has determinant equal to the determinant of $V$. 
\end{proof}}

\begin{remark}
{From Proposition \ref{prop:strucproperty4DCkernel}, it is easy to see  that $\cK_{DC}^{-1}$ admits the Cholesky factorization
\begin{align}\label{eq:invchol} \cK_{DC}^{-1}=DD^\top
\end{align} where  $D$ is an $n$-dimensional lower bidiagonal matrix:
\begin{equation*}
D(i,j) =
\begin{cases}
\lambda^{-i/2}(1-\rho^2)^{-1/2}, & \text{for } i=j, j < n \\
- \lambda^{-i/2}(1-\rho^2)^{-1/2} \lambda^{-1/2}\rho,  & \text{for } i=j+1, j < n \\
\lambda^{-n/2} & \text{for } i=j=n
\end{cases}\,.
\end{equation*}}
\end{remark}

\begin{remark} 
{TC kernels can be seen as a special case of DC kernels with $\rho=\sqrt{\lambda}$.
It follows that closed-form expressions for the factors of the TC kernel, as well as for its determinant and inverse,
can be obtained by setting $\rho=\sqrt{\lambda}$ in \eqref{eqn:fatt_KDC}--
\eqref{eqn:det_K_DC}, see \cite{Carli2014} for a complete derivation.
Nevertheless, as mentioned above, the TC kernel can also be seen as a particular member of the family of stable spline kernels (stable spline kernel of order one -- see \cite{PillonettoDeNicolao2010}).
However, it is straightforward to verify that 
the maximum entropy interpretation does \emph{not} extend to the family of stable spline kernels
(e.g. to second--order stable spline kernels), that do \emph{not} admit a tridiagonal inverse.}
\end{remark}


\section{Exploiting the  DC kernel structure for solving (\ref{eq:lglklhd}) }\label{sec:comp_complexity_DC}

{
As pointed out in Remark \ref{rem:marginal-lik-maximization_discussion},  kernel matrices can be very ill--conditioned and thus
special care needs to be taken in solving the marginal likelihood maximization problem (\ref{eq:lglklhd}) (and also (\ref{eqn:f})).
This issue has been discussed in \cite{ChenLjung2013} and the algorithm proposed therein relies on the numerical computation of the Cholesky factorization of $\cK_{DC}$.
{In fact, a definite improvement both of \emph{numerical stability} and 
\emph{computational efficiency} can be achieved by exploiting 
sparsity of the factorization of $\cK_{DC}^{-1}$ along with its closed--form representation.
The modified algorithm is described in the following. }
}




\subsection{Making use of the structure of DC kernel}
By exploiting the Sylvester's determinant theorem and the matrix inversion lemma, the addends in the objective \eqref{eq:lglklhd} can be rewritten as
\begin{align}\nonumber
&\det (\sigma^2I_N + \Phi^\top\cK_{DC}\Phi)\\\nonumber &\qquad
\qquad = (\sigma^2)^{N-n}\det (\cK_{DC}) \det(\sigma^2 \cK_{DC}^{-1}
+\Phi\Phi^\top)\\
&Y^\top(\sigma^2I_N + \Phi^\top\cK_{DC}\Phi)^{-1}Y  \nonumber \\
  &  \qquad \qquad  = \frac{Y^\top Y}{\sigma^2}- Y^\top\Phi^\top(\sigma^2\cK_{DC}^{-1} + \Phi\Phi^\top)^{-1}\frac{\Phi Y}{\sigma^2}
\label{eq:lglklhd_reduce}
\end{align}
{where the term
$(\sigma^2\cK_{DC}^{-1} + \Phi\Phi^\top)^{-1}\frac{\Phi Y}{\sigma^2}$
in  \eqref{eq:lglklhd_reduce} can be seen as the solution of the least square problem
$
\min_{x} \left\| Ax -b\right\|^{2}
$
with coefficient matrix $A$ and vector $b$ given by
$A = \bmat \Phi & \sigma D \emat^\top$, 
$b= \bmat Y^{\top} & 0 \emat^{\top}$. }

{Now, recall the definition of \emph{thin} QR
factorization \cite{GVL:96}: if $B=CD$ is a QR factorization of
$B\in\R^{p\times q}$ with $p\geq q$, then $B=C(1:p,1:q)D(1:q)$ is
called the thin QR factorization of $B$, where, without loss of
generality, $D(1:q)$ is assumed to have positive diagonal entries
\cite{GVL:96}.} Then consider the thin QR factorization
of
\begin{align}\label{eq:qr1} \left[
  \begin{array}{cc}
    \Phi^\top & Y \\
    \sigma D^\top & 0 \\
  \end{array}
\right] = QR = Q\left[
              \begin{array}{cc}
                R_1 & R_2 \\
              0 & r\end{array}
            \right]
\end{align}
where $D$ is the Cholesky factor in \eqref{eq:invchol},
$Q$ is an $(N+n)\times (n+1)$ matrix such $Q^\top Q=I_{n+1}$,
and $R$ is an $(n+1)\times (n+1)$ upper triangular matrix.
Here, $R$ is further partitioned into $2\times2$ blocks with $R_1,R_2$ and $r$ being an $n\times n$ matrix, an $n\times 1$ vector and a scalar, respectively.
By left--multiplying both sides of \eqref{eq:qr1} by its transpose, taking into account that $Q^\top Q=I_{n+1}$, we have
\begin{align*} &\sigma^2
\cK_{DC}^{-1} +\Phi\Phi^\top  = R_1^\top R_1, \quad \Phi Y = R_1^\top R_2,\quad
Y^\top Y = R_2^\top R_2 + r^2,
\end{align*}
so that terms in \eqref{eq:lglklhd_reduce} can be written as
\begin{align*}
&\det (\sigma^2I_N + \Phi^\top\cK_{DC}\Phi) 
=(\sigma^2)^{N-n} \det (\cK_{DC}) \det (R_1)^2\,,\\
&Y^\top(\sigma^2I_N +
\Phi^\top\cK_{DC}\Phi)^{-1}Y 
= r^2/\sigma^2\,.
\end{align*}
Further, by exploiting the closed form expression of the determinant of $\cK_{DC}$ (\ref{eqn:det_K_DC}), the cost function in
(\ref{eq:lglklhd}) can be rewritten as
\begin{align}\label{eq:lglklhd2}  
&\frac{r^2}{\sigma^2} + (N-n)\log\sigma^2 + \frac{n(n+1)}{2}\log
\lambda   \nonumber \\&\qquad\qquad+ (n-1)\log(1-\rho^2)+  2\log
\det (R_1) \,.
\end{align}

\subsection{Preprocessing the data}

{The computational complexity associated with the QR factorization (\ref{eq:qr1}) depends on the number of data points $N$.
We can get rid of this dependency in the following way. }
Consider the thin QR factorization 
\begin{align}\label{eq:qr0} \left[
  \begin{array}{cc}
    \Phi^\top & Y \\
  \end{array}
\right] =Q_d \left[
  \begin{array}{cc}
    R_{d1} & R_{d2} \\
  \end{array}
\right]
\end{align}
where $Q_d$ is an $N\times (n+1)$ matrix whose columns are orthogonal unit vectors
such that $Q_d^\top Q_d=I_{n+1}$, $R_{d1}$ is an $(n+1)\times n$ matrix and $R_{d2}$ is an $(n+1)\times 1$ vector,
and the thin QR factorization of
\begin{align}\label{eq:qr4} \left[
  \begin{array}{cc}
    R_{d1} & R_{d2} \\
    \sigma D^\top & 0 \\
  \end{array}
\right] =  Q_c R_c
\end{align}
where $ Q_c$ is an $(2n+1)\times (n+1)$ matrix 
such that $Q_c^\top Q_c=I_{n+1}$ and $R_c$ is an $(n+1)\times (n+1)$ upper triangular matrix.
Then from (\ref{eq:qr0}) and (\ref{eq:qr4}), we have
\begin{align}\label{eq:qr5} \left[
  \begin{array}{cc}
    \Phi^\top & Y \\
    \sigma D^\top & 0 \\
  \end{array}
\right] = \left[
              \begin{array}{cc}
                Q_d & 0 \\
              0 & I_n\end{array}
            \right]   Q_cR_c\,,
\end{align}
{which actually gives another thin QR factorization
of $\left[
  \begin{array}{cc}
    \Phi^\top & Y \\
    \sigma D^\top & 0 \\
  \end{array}
\right]$ and where $R_{c}$ can be block--partitioned similarly to what done in \eqref{eq:qr1}. }



\begin{remark} 
{
One may wonder whether $\left[
  \begin{array}{cc}
    \Phi^\top & Y \\
    \sigma D^\top & 0 \\
  \end{array}
\right]$ can have a unique thin QR factorization.
The answer is affirmative if $[\Phi^{\top} \, Y]$ is full column
rank according to \cite[Thm 5.2.2]{GVL:96}.}


\end{remark}

\subsection{Proposed algorithm}

\begin{proposition}\label{prop} 
{The cost function in (\ref{eq:lglklhd}) can be
computed according to (\ref{eq:lglklhd2}).}
For a given $\eta$, the MAP 
estimate (\ref{eqn:f}) can be computed as 
\begin{align}
  \hat g_{MAP} = (\sigma^2 \cK_{DC} ^{-1}
+\Phi\Phi^\top)^{-1}\Phi Y = R_1^{-1} R_2
\end{align}
\end{proposition}

\begin{algorithm}\label{alg4}
Assume that the thin QR factorization (\ref{eq:qr0}) has been computed beforehand.
Then the cost function in (\ref{eq:lglklhd}) can be computed via the following steps:
\begin{itemize}
\item[1)] {evaluate the Cholesky factor $D$ of $\cK_{DC}^{-1}$ 
according to \eqref{eq:invchol}. }
\item[2)] compute the thin QR factorization (\ref{eq:qr4});
\item[3)] compute the objective function according to (\ref{eq:lglklhd2}). 
\end{itemize}
\end{algorithm}


{Clearly, Algorithm \ref{alg4} is numerically more
stable than Algorithm 2 in \cite{ChenLjung2013} because it 
makes use of the closed--form Cholesky factors of $\cK_{DC}^{-1}$
while Algorithm 2 in \cite{ChenLjung2013} relies on the numerical computation of the Cholesky factors of $\cK_{DC}$, that can be very ill--conditioned. 
}

Now we compare  Algorithm \ref{alg4} with Algorithm 2 in 
\cite{ChenLjung2013} in terms of computational complexity. Both
algorithms consist of three parts:
\begin{itemize}
\item \textit{preprocessing}:
Both algorithms
require to compute (\ref{eq:qr0}) beforehand, 
with a computational cost of $2(n+1)^2(N-(n+1)/3)$ flops (see \cite{GVL:96}).
This part is done once at the beginning of marginal likelihood maximization process.


\item \textit{preparation for the evaluation of (\ref{eq:lglklhd})}:
Algorithm \ref{alg4} needs to create $D^\top$, the Cholesky
factorization of $\cK_{DC} ^{-1}$, and  Algorithm 2 in
\cite{ChenLjung2013} requires to create $\cK_{DC}$. Since $D^\top$
is bidiagonal, only $2n-1$ scalars have to be created. In contrast,
since $\cK_{DC}$ is symmetric, $n(n+1)/2$ scalars are created. This
part is done at every iteration in the solution of the marginal
likelihood maximization.


\item \textit{evaluation of (\ref{eq:lglklhd})}:
For what concerns Algorithm \ref{alg4}, a straightforward computation of the QR
factorization (\ref{eq:qr4}) in step 2) requires
$2(n+1)^2(2n+1-(n+1)/3)$ flops, and step 3) requires $n+20$ flops. 
For what concerns Algorithm 2 in \cite{ChenLjung2013}, steps 1) (\emph{numerical computation} of the Cholesky factorization of $\cK_{DC}$) and 4) (evaluation of the cost function),
require $n^3/3+n^2/2+n/6$ and $2n+6$ flops, respectively  \cite{Hunger:07}. 
Step 2) (matrix multiplication) requires $n^2(n+1)$ flops and
step 3) (QR factorization) 
requires $2(n+1)^2(2n+1-(n+1)/3)$ flops \cite{GVL:96}. 
This part is done at every iteration in the solution of the marginal
likelihood maximization.


\end{itemize}

To summarize, for the \emph{preprocessing}, both algorithms require
the same amount of computations. For the \emph{preparation for the
evaluation of (\ref{eq:lglklhd})} with large $n$,  the storage and
computational complexity of Algorithm \ref{alg4} is negligible if
compared to the one of Algorithm 2 in \cite{ChenLjung2013}. Finally,
for the \emph{evaluation of (\ref{eq:lglklhd})} with large $n$,
Algorithm \ref{alg4} saves approximately 28\% of the flops required by Algorithm 2 in \cite{ChenLjung2013}. 
The computational time required by the two algorithms is compared in Example \ref{exm} below. 

\begin{remark} 
{
When optimizing the marginal likelihood, the dimension of the optimization variable (the hyperparameters vector) is usually small (e.g. $\eta = [c, \lambda, \,\rho]$ for the DC kernel),
so marginal likelihood maximization can be performed 
by using some derivative--free nonlinear optimization solver, such as \texttt{fminsearch} in Matlab,
combined with Algorithm \ref{alg4} for the  computation of the cost function \eqref{eq:lglklhd}. 
On the other hand, there are many gradient and/or Hessian based nonlinear optimization solvers (e.g. \texttt{fmincon} in Matlab), which may be used to speed up the optimization of \eqref{eq:lglklhd}.
} Let $l(\cK_{DC})$ denote the cost function in (\ref{eq:lglklhd}). Then its gradient and Hessian are computed as:
\small\begin{align*}
\frac{\partial l(\cK_{DC})}{\partial \eta_i} &= \rm{trace}( (X_1 -
X_2) \frac{\partial
\cK_{DC}}{\partial \eta_i}) \\
\frac{\partial^2 l(\cK_{DC})}{\partial \eta_i\eta_j} & = \rm{trace}(
(X_1 - X_2) \frac{\partial^2 \cK_{DC}}{\partial
\eta_i\eta_j} \nonumber \\
&+ \rm{trace}( (X_1\frac{\partial \cK_{DC}}{\partial \eta_i}X_2  -
(X_1-X_2) \frac{\partial \cK_{DC}}{\partial \eta_i} X_1  )
\frac{\partial \cK_{DC}}{\partial
\eta_j})
\end{align*}\normalsize where $\eta_i$ the $i$th element of $\eta$, and $X_{1}$ and $X_{2}$ are given by:
$X_1 =  \cK_{DC}^{-1} - \sigma^2 \cK_{DC}^{-1} (R_1^\top R_1)^{-1}\cK_{DC}^{-1}$, $ X_2 =  \cK_{DC}^{-1}R_1^{-1}R_2R_2^\top R_1^{-T} \cK_{DC}^{-1}$.
Thus, by exploiting closed form expression and the sparse tridiagonal structure of $\cK_{DC}^{-1}$, the gradient and Hessian of the cost function \eqref{eq:lglklhd} 
can be computed in an efficient and robust way.
\end{remark}

\begin{remark} \label{rmkadpalg}
Algorithm 2 in \cite{ChenLjung2013} can be modified by making use of the Cholesky factorization of $\cK_{DC}$,
$\cK_{DC}=LL^\top$ with $L=UW^{1/2}$ (see \eqref{eqn:fatt_KDC}). 
The adapted algorithm does not involve numerical computation of the Cholesky factorization.
The computational saving that can be achieved via this modification is discussed in the following example.
\end{remark}

\begin{example}\label{exm}
{
We show that Algorithm \ref{alg4} is computationally more efficient
than Algorithm 2 in \cite{ChenLjung2013}.
We set the FIR model order $n=125$, the number of data point $N=500$, and the hyperparameters of the DC kernel (\ref{eqn:DC_kernel}) to $c=1,\lambda=0.9,\rho=0.8$ and
we generate the data $\Phi$ and $Y$ randomly and take the noise variance $\sigma^2=0.2$.
Then we compute the cost function in (\ref{eq:lglklhd}) for 5000 times with the above specified parameters and data.
We test three algorithms denoted by \texttt{A}, \texttt{B}, \texttt{C}, respectively: algorithm \texttt{A} is Algorithm 2 in \cite{ChenLjung2013},
algorithm \texttt{B} is the adapted algorithm sketched in Remark \ref{rmkadpalg}, and  algorithm \texttt{C} is Algorithm \ref{alg4}.
The time (in seconds) taken by different algorithms for 5000 evaluations of the cost function in (\ref{eq:lglklhd}) 
is given below:
\begin{center}
  \begin{tabular}{ l  || c | c | c }
    \hline
    Algorithm & \texttt{A} & \texttt{B} & \texttt{C} \\ \hline
    Time [sec] & 21.9 & 20.4 & 6.8 \\ \hline
  \end{tabular}
\end{center}
which shows 
that algorithm \texttt{B} and algorithm \texttt{C} save, respectively,  6.9\% and 69.2\% of the computational time required by
algorithm \texttt{A}. }


\end{example}

%


\vspace{-2mm}

\section{Conclusions}\label{sec:Conclusions}

Gaussian process regression for linear system identification has become popular recently, due to the introduction
of families of kernels which encode structural properties of the dynamical system.
A family of kernels that has been shown to be
particularly effective 
is the family of 
DC kernels.
The main contribution of the present paper is to show that the family of DC kernels admits a maximum entropy interpretation.
{This interpretation can be exploited in conjunction with results on matrix completion problems in the graphical models literature
to shed light on 
the structure of the DC kernel. }
In particular, we proved that the DC kernel admits a closed-form inverse, determinant and factorization.
The DC kernel structure has been exploited to improve the numerical stability and reduce computational complexity associated with the computation of the associated estimator.




%
%
%
%
%
%
%
%

\bibliographystyle{plain}
\bibliography{biblio_DC_kernel}


\end{document}